\documentclass[11pt, reqno]{amsart}

\usepackage{amssymb,latexsym,amsmath,amsfonts}
\usepackage{mathrsfs}
\usepackage{graphicx}

\numberwithin{equation}{section}

\theoremstyle{definition}
\newtheorem{definition}{Definition}[section]

\theoremstyle{remark}
\newtheorem{remark}[definition]{Remark}

\theoremstyle{plain}
\newtheorem{theorem}[definition]{Theorem}
\newtheorem{result}[definition]{Result}
\newtheorem{lemma}[definition]{Lemma}
\newtheorem{proposition}[definition]{Proposition}
\newtheorem{example}[definition]{Example}

\newcommand{\eps}{\varepsilon}
\newcommand{\zt}{\zeta}

\newcommand{\tht}{\theta}

\newcommand{\bdy}{\partial}
\newcommand{\OM}{\Omega}
\newcommand{\D}{\mathbb{D}}

\newcommand{\smoo}{\mathcal{C}}
\newcommand{\hol}{\mathcal{O}}
\newcommand\leb[1]{\mathbb{L}^{{#1}}}

\newcommand{\uni}{{\sf U}}

\newcommand{\re}{{\sf Re}}
\newcommand{\im}{{\sf Im}}
\newcommand{\bcdot}{\boldsymbol{\cdot}}
\newcommand\cis[1]{e^{i{#1}}}
\newcommand{\lrarw}{\longrightarrow}

\newcommand{\poi}{{\sf p}}
\newcommand\flt[1]{\varPsi_{{#1}}}
\newcommand{\z}{z^{\raisebox{-1pt}{$\scriptstyle {\prime}$}}}

\newcommand{\V}{V^{\raisebox{-1pt}{$\scriptstyle {\prime}$}}}

\newcommand\pth[1]{\gamma^{\raisebox{-2pt}{$\scriptstyle{{#1}}$}}}

\newcommand\crux[1]{\left[\log\left(\frac{C}{{#1}}\right)\right]^{-1/\alpha}}
\newcommand\CruxII[2]{\left[\log\left(\frac{{#2}}{{#1}}\right)\right]^{-1/\alpha}}
\newcommand{\suppmf}{\boldsymbol{{\sf M}}}
\newcommand\bv[1]{{#1}^{\bullet}}
\newcommand\omo[1]{\omega_{{#1}}}
\newcommand\cnj[1]{\widetilde{{#1}}}
\newcommand\pw[1]{\text{\small ${#1}$}}

\newcommand{\Cn}{\mathbb{C}^n}
\newcommand{\C}{\mathbb{C}} 
\newcommand{\R}{\mathbb{R}}

\newcommand{\N}{\mathbb{N}}

\begin{document}

\title[Complex geodesics, Hardy--Littlewood, etc.]{Complex geodesics, their boundary regularity, \\
 and a Hardy--Littlewood-type lemma}

\author{Gautam Bharali}
\address{Department of Mathematics, Indian Institute of Science, Bangalore 560012, India}
\email{bharali@math.iisc.ernet.in}

\thanks{This work is supported in part by a UGC Centre for Advanced Study grant. \\
${ \; }\;\,{ \; }$ To appear in {\bf{\em Ann. Acad. Sci. Fennic{\ae}.\,Math}}}


\begin{abstract}
{We begin by giving an example of a smoothly bounded convex domain that has complex geodesics
that do not extend continuously up to $\bdy\D$. This example suggests that continuity at the
boundary of the complex geodesics of a convex domain $\OM\Subset \Cn$, $n\geq 2$, is affected by
the extent to which $\bdy\OM$ curves or bends at each boundary point. We provide a sufficient condition
to this effect (on $\smoo^1$-smoothly bounded convex domains), which admits domains having boundary points
at which the boundary is infinitely flat. Along the way, we establish a Hardy--Littlewood-type lemma
that might be of independent interest.}
 
\end{abstract}
\maketitle

\section{Introduction}\label{S:intro}

Let $\OM$ be a bounded domain in $\Cn$ and let $\D$ denote the open unit disc centered
at $0\in \C$. A holomorphic map $f: \D\to \OM$ is called a {\em complex geodesic} of
$\OM$ if it is an isometry for the Kobayashi distances on $\D$ and $\OM$ (since $\OM$ is bounded, the
Kobayashi pseudodistance on $\OM$ is a true distance). These objects provide the primary motivation for this
work. Along the way, we prove a result in one complex variable that arose from our need for a type of
Hardy--Littlewood lemma on $\D$. Since the latter
topic is familiar to a large number of readers, we defer its discussion to Section~\ref{S:HL}.
\smallskip 

A fundamental theorem about the existence of complex geodesics is the following result:

\begin{result}[Lempert, \cite{lempert:mKrdb81}]\label{R:lempert}
Let $\OM$ be a bounded strictly convex domain in $\Cn$ with $\smoo^3$-smooth boundary.
\begin{enumerate}
 \item[$a)$] Given any two distinct points $z_1, z_2\in \OM$, there exists a complex geodesic
 of $\OM$ whose image contains $z_1$ and $z_2$.
 
 \item[$b)$] If, furthermore, $\OM$ is strongly convex, then every complex geodesic
 $f: \D\to \OM$ extends to a map of class $\smoo^1(\overline{\D})$.
\end{enumerate}
\end{result}

\noindent{A domain $\OM$ is said to be {\em strictly convex} if for any two points $z_1, z_2\in
\overline\OM$ the open segment $\{tz_1+(1-t)z_2: 0 < t < 1\}\subset \OM$. The strongly
convex domains form a proper subclass of the class of strictly convex domains: a convex
domain $\OM$ is said to be {\em strongly convex} if it has a $\smoo^2$-smooth boundary
and the second fundamental form of $\bdy\OM$ is strictly positive definite.}
\smallskip

The analogue of part\,$(a)$ of Result~\ref{R:lempert} for all convex domains in general was
established by Royden and Wong \cite{roydenWong:CKmcd83}. The
Royden--Wong extension of Lempert's theorem does not, however, make any assertions
about the boundary regularity of complex geodesics. 
\smallskip

One {\em cannot}, in general, expect the complex geodesics of a convex domain to extend
even continuously up to $\bdy\D$. To see this, consider the polydisc $\D^n$, $n\geq 2$: it is
easy to see that $f = (f_1,\dots,f_n): \D\to \D^n$ is a complex geodesic if and only if
at least one of $f_1,\dots, f_n$ is an automorphism of $\D$. By choosing any one map
among $f_1,\dots, f_n$ to be such that it does not extend continuously up to $\bdy\D$,
we see that complex geodesics do not, in general, extend continuously up to $\bdy\D$.
This example might lead one to suspect that the non-smoothness of $\bdy\OM$ is 
the chief reason that a complex geodesic does not extend
continuously up to $\bdy\D$. However, non-smoothness of $\bdy\OM$ is not the relevant
issue, as the following example shows:

\begin{example}
An example of a bounded convex domain with $\smoo^\infty$-smooth boundary having
complex geodesics that do not extend continuously to $\bdy\D$.
\end{example}

\noindent{Consider the complex geodesic of $\D^2$, $f=(f_1,f_2)$, where $f_1$ is an
automorphism of $\D$ and $f_2$ is a bounded holomorphic function with $|f_2| < 1/2$
that {\em does not} extend continuously up to $\bdy\D$. Let $\OM$ be any convex
domain having $\smoo^\infty$-smooth boundary such that 
$\D\times D(0;1/2)\subset \OM\varsubsetneq \D^2$. $K_{G}$ will denote
the Kobayashi distance on the domain $G$ and $\poi$ the Poincar{\'e} metric. Then: \vspace{-1mm}
\begin{align*}
 \poi(\zt_1,\zt_2)\,&=\,K_{\D^2}(f(\zt_1), f(\zt_2)) \\
 				&\leq\,K_{\OM}(f(\zt_1), f(\zt_2))\,\leq\,\poi(\zt_1,\zt_2) \quad
 				\forall \zt_1, \zt_2\in \D. \vspace{-1mm}
\end{align*}
The equality above encodes the fact that $f$ is a complex geodesic of $\D^2$ while the
first inequality is the distance-decreasing effect of the inclusion that maps $\OM\hookrightarrow \D^2$ .
Thus, $f$ is a complex geodesic of $\OM$, but it does not extend continuously up to
$\bdy\D$. \hfill $\blacktriangleleft$}
\vspace{1mm}

This example suggests that the property of a convex domain $\OM$ that affects the boundary
behaviour of {\em generic} complex geodesics is the flatness of $\bdy\OM$ or the extent to which $\bdy\OM$
curves or bends at each boundary point. This notion is supported by part\,$(b)$ of Result~\ref{R:lempert}.
(We emphasize the word ``generic'' here because even in domains with rough boundaries
there may exist points $z_1$ and $z_2$ in special position, consider $\D^2$ for instance, such that some
complex geodesic containing
them extends continuously up to $\bdy\D$. We will not address this type of non-generic phenomena in
this work.) Our notion is further supported  by a result of Mercer \cite[Proposition~2.9]{mercer:cgihmcd},
which states that all complex geodesics of a bounded $m$-convex domain\,---\,see Definition~2.7
in \cite{mercer:cgihmcd}\,---\,extend to maps that are H{\"o}lder-continuous on
$\overline\D$ (where the H{\"o}lder exponent depends on the parameter $m$).
\smallskip

If a domain $\OM$ is smoothly bounded and $m$-convex, then for each
$w\in \bdy\OM$ the (complex) order of contact of the complex line
$w + \C{v}$, for each $v\in H_w(\bdy\OM) := 
T_w(\bdy\OM)\cap iT_w(\bdy\OM), \ v\neq 0$, with $\bdy\OM$ at $w$ is at most $m$.
In contrast, we will show that all complex geodesics of
$\OM$ extend continuously up to $\bdy\D$ even if there are points $w\in \bdy\OM$ at which
$w + \C{v}$, for
some $v\in H_w(\bdy\OM)$, osculates $\bdy\OM$ to {\em infinite} (complex) order, provided
$\bdy\OM$ exhibits some degree of bending in the complex-tangential directions. To make this precise
when $\bdy\OM$ is merely $\smoo^1$-smooth requires a little effort. To this end, we
need the following definition. (Henceforth, $B^{d}(a;r)$ will denote the Euclidean
ball in $\C^d$ with centre $a$ and radius $r$.)

\begin{definition}\label{D:supp}
Let $\OM$ be a bounded convex domain in $\Cn, \ n\geq 2$, with $\smoo^1$-smooth boundary.
Let $F: B^{n-1}(0;r)\to \R$ be a smooth function with $F(0)=0$ and $DF(0)=0$. We say that
{\em $F$ supports $\OM$ from the outside} if there exist constants $R_0\in (0,r)$ and
$s_0 > 0$ such that, for each $w\in \bdy\OM$, there exists a unitary transformation $\uni_w$
satisfying:
\begin{itemize}
 \item $\uni_w(H_w(\bdy\OM)) = \{v\in \Cn: v_n=0\}$, and
 \item $\uni_w(\nu_w) = (0,\dots,0,i)$, where $\nu_w$ denotes the inward unit normal
 vector to $\bdy\OM$ at $w$,
\end{itemize}
such that, denoting the $\C$-affine map $v\!\longmapsto\!\uni_w(v-w)$ as $\uni^w$, we have\vspace{-2.0mm}
\begin{multline*}
 \uni^w(\overline\OM)\cap B^{n-1}(0;R_0)\times((-s_0,s_0)+i(-s_0,s_0)) \\
 \subset\,\{z=(\z,z_n)\in
 B^{n-1}(0;R_0)\times((-s_0,s_0)+i(-s_0,s_0)): \im(z_n) \geq F(\z)\}.
\end{multline*}
\end{definition}

Perhaps the most familiar examples of functions on $[0,\infty)$ that vanish to infinite order at $0$
are the functions $\flt{\alpha}$, $\alpha > 0$, defined as follows:
\begin{equation}\label{E:good}
 \flt{\alpha}(x)\,:=\,\begin{cases}
 			e^{-1/x^{\alpha}},	&\text{if $x > 0$}, \\
 			0,			&\text{if $x = 0$}.
			\end{cases}
\end{equation}
These functions help us to translate the qualitative notion expressed prior to Definition~\ref{D:supp}
to give us the following result:

\begin{theorem}\label{T:contEXT}
Let $\OM$ be a bounded convex domain in $\Cn, \ n\geq 2$, with $\smoo^1$-smooth boundary.
Assume that $\OM$ is supported from the outside by $F(\z) := 
C\flt{\alpha}(\,\|\z\|^\alpha)$ (writing $z = (\z, z_n)$ for each $z\in \Cn$) for some $C > 0$
and some $\alpha\in (0,1)$. Then, every complex geodesic of $\OM$ extends continuously up
to $\bdy\D$.
\end{theorem}

\begin{remark}
Note that a domain that satisfies the hypothesis of Theorem~\ref{T:contEXT} need not be
strictly convex. For such a domain $\OM$, it is possible for $\bdy\OM$ to contain line segments
that point along the complex-normal direction at each of the boundary points through which they pass.
\end{remark}

\begin{remark}
The condition in Theorem~\ref{T:contEXT} might seem at first a bit artificial. However,
readers familiar with the analysis of the boundary geometry of a domain around points that are
{\em not} of finite type know that {\em there exists no classification of the local normal forms} for
$\bdy\OM$, at infinite-type points in $\bdy\OM$, analogous to even the very
general notion due to Catlin \cite{catlin:bipd84}. When $\bdy\OM$ contains points at which it is
infinitely flat and yet $\bdy\OM$ is assumed to have {\em low regularity globally}, a
way to model $\bdy\OM$ is through objects such as the one in Definition~\ref{D:supp}. Even then,
choices need to be made. We have chosen the functions in \eqref{E:good} to underlie the functions
supporting $\OM$ because these are the most well-known functions
that vanish to infinite order.
\end{remark}

What can one say if the convex domain given is supported from the outside by $C\flt{\alpha}(\,\|\bcdot\|\,)$
and $\alpha\geq 1$\,? This is a subtle question, but see Remark~\ref{R:aboutAZ} below. In working with the
latter functions, $\alpha = 1$ usually marks a transition-point for
methods that work well for $0 <\alpha < 1$. For illustrations of this, see, for instance,
\cite[Section~3]{fornaessLee:abKmnd09},
\cite[Section~3]{halfpapNagelWainger:BSknpit10} or \cite{fornaessLeeZhanh:sedbaritcdC^2:11}.
\smallskip  
    
The key quantitative ingredient in the proof of Theorem~\ref{T:contEXT} is a
simple extension of one of the Hardy--Littlewood lemmas to holomorphic functions on $\D$ whose
derivatives have rather rapid growth. This extension leads to a characterization (which is unrelated
to Theorem~\ref{T:contEXT}, but may be of independent interest)
of a class of holomorphic functions that is {\em strictly larger} than every class
of holomorphic functions on $\D$ having a H{\"o}lder-continuous extension to $\bdy\D$.
We discuss these matters in the next section. The proof of Theorem~\ref{T:contEXT}
will be given in Section~\ref{S:contEXT}.

\begin{remark}\label{R:aboutAZ}
A few months after this paper was written, Zimmer showed, among other things, that if $\OM$
is a $\C$-strictly convex domain having $\smoo^{1, \alpha}$-smooth boundary, then every
complex geodesic of $\OM$ extends continuously up to $\bdy\D$ 
\cite[Corollary~1.8]{zimmer:cpdag15}. The domains in our
Theorem~\ref{T:contEXT} are, in Zimmer's terminology, $\C$-strictly convex, although
\cite[Corollary~1.8]{zimmer:cpdag15} addresses only those domains in Theorem~\ref{T:contEXT}
that have $\smoo^{1, \alpha}$-smooth boundaries. However, a $\C$-strictly convex $\OM$
admits points $p\in \bdy\OM$ at which $\bdy\OM$ can be flat to any extent without containing
a germ of a complex line at $p$. Zimmer's proof uses ingredients very different from those alluded
to above. The constraint that $\bdy\OM$ be $\smoo^{1, \alpha}$-smooth arises from one of
those ingredients.
\end{remark}
 
\section{A Hardy--Littlewood-type lemma}\label{S:HL}

The phrase ``Hardy--Littlewood-type lemma'' refers in the present context to any type of result that, given
a function $f\in \hol(\D)$:\vspace{-0.5mm}
\begin{itemize}
 \item[$a)$] tells us, based on the growth of $|f^\prime(\zt)|$ in terms of ${\sf dist}(\zt, \bdy\D)$,
 whether $f$ extends continuously up to $\bdy\D$,
 \item[{}] {\em AND}
 \item[$b)$] if possible, characterizes the extension of $f$ to $\bdy\D$ in terms of its modulus
 of continuity.
\end{itemize}
\vspace{1mm}

\noindent{We shall present here a result of this type. Our result arises from the following 
proposition, which is central to proving Theorem~\ref{T:contEXT}. Some notation: given an
interval $E\subset \R$, $\leb{1}(E)$ will denote the $\leb{1}$-class with respect to the
Lebesgue measure on $E$.

\begin{proposition}\label{P:HL}
Let $\Phi: [0,r_0)\to [0,+\infty]$ be a function
of class $\leb{1}([0,r_0))$ for some $r_0\in (0, 1)$. Let $f\in \hol(\D)$ and assume that
\[
 |f^\prime(r\cis{\tht})|\,\leq\,\Phi(1-r) \quad \forall r: 1-r_0 < r < 1 \ \text{and} \ \forall\tht\in \R.
\]
Then, $f$ extends continuously to $\bdy\D$.
\end{proposition}

\noindent{The proof of this statement is simple, and we shall skip some elementary details. The condition on
$f^\prime$ implies, owing to the dominated convergence theorem, that the limit
\begin{equation}\label{E:bLim}
 \bv{f}(\cis{\tht})\,:=\,\lim_{r\to 1^-}f(r\cis{\tht})\,=\,f(0) +
 \lim_{r\to 1^-}\,\cis{\tht}\!\int_0^r\!f^{\prime}(s\cis{\tht})ds
\end{equation}
exists for each $\tht\in \R$, and that this limit is {\em uniform} in $\tht$. This, together with the
fact that, for a fixed $R\in (0,1)$, we can make $|f(R\cis{\tht_1}) - f(R\cis{\tht_2})|$ as small as we
wish by taking $\tht_1$ and $\tht_2$ sufficiently close to each other, implies that $\bv{f}\in \smoo(\bdy\D)$.
The usual Poisson-integral argument establishes that $\bv{f}(\cis{\tht})$ is the continuous extension of $f$
to $\cis{\tht}\in \bdy\D$. 
\smallskip

To proceed further, we must recall a definition. Given a function $g: \bdy\D\lrarw \C$,
we define the 
{\em modulus of continuity of $g$} by
\[
 \omo{g}(\delta)\,:=\,\sup_{|\tht-\phi|\leq\delta}|g(\cis{\tht})-g(\cis{\phi})|,  \ \ \ 0\leq \delta\leq \pi.
\]
This concept can be defined in a much more general setting, but we shall restrict ourselves to
$\C$-valued functions on $\bdy\D$. Clearly
\begin{equation}\label{E:limit0}
 g\in \smoo(\bdy\D;\,\C) \iff \lim_{\delta\to 0^+}\omo{g}(\delta) = 0.
\end{equation}

The classical Hardy--Littlewood lemma characterizes, in terms of
the growth of $f^\prime$, functions $f\in \hol(\D)$ that extend continuously up to $\bdy\D$ such that
$\omo{\bv{f}}(\delta) = C\delta^{\alpha}$, $C > 0$ and $0< \alpha\leq 1$, i.e., such that 
the boundary-value of $f$ belongs to a H{\"o}lder class on $\bdy\D$. In this section, given
$f\in \hol(\D)$, $\bv{f}$ will denote the radial boundary-value of $f$ (whenever it exists).
\smallskip

Since very little is assumed about the function $\Phi$ in Proposition~\ref{P:HL}, characterizing the
boundary-values $\bv{f}$ of the functions $f$ mentioned therein is probably not tractable. But it
does raise a natural question: is there such a characterization of a class that:\vspace{-0.5mm}
\begin{itemize}
 \item[1)] includes, for instance, functions for which $\omo{\bv{f}}(\delta) = (\log(1/\delta))^{-1}$ for $\delta$
 close to $0$;\vspace{-0.5mm}
 \item[{}] or 
 \item[2)] {\em at least} includes the case $\omo{\bv{f}}(\delta) = e^{\text{\small{$-\alpha(\log(1/\delta))^{1-\eps}$}}}$,
 $0< \eps< 1$, for $\delta$ close to $0$\,?
\end{itemize}
Observe: when $\eps$ is close to $0$ and $0< \alpha< 1$, the values of the latter function are very close to
$\delta^\alpha$ for $\delta\leq 1$, unless $\delta$ is {\em extremely} small. Yet, if there are any functions
$f\in \hol(\D)$ for which (2) is true, then $\bv{f}$ would be rougher than any H{\"o}lder-continuous
function.
\smallskip

In the recent literature, we have been introduced to classes of functions\,---\,defined by their
moduli of continuity\,---\,that are not as rough as the boundary-values given by Proposition~\ref{P:HL}
but are considerably less regular than the H{\"o}lder-continuous functions: see,
for instance, \cite{dyakonov:enL-tshf97} by Dyakonov, or \cite{kuusiMingione:upe12} by
Kuusi--Mingione. The next definition describes a continuity class in the style of the latter paper, but
which is large enough to include functions having the modulus of continuity described in (2).
\smallskip

\begin{definition}
Let $g\in \smoo(\bdy\D;\,\C)$. We say that $g$ is {\em log-Dini continuous} if $\omo{g}$, the
modulus of continuity of $g$, satisfies
\begin{equation}\label{E:logDini}
 \int_{0}^1(\log(1/x))^n\,\frac{\omo{g}(x)}{x}\,dx\,<\,\infty \ \ \ \text{for} \ n = 0, 1, 2,\dots
\end{equation}
\end{definition}

We are now in a position to state the main theorem of this section.

\begin{theorem}\label{T:HL}
Let $f\in \hol(\D)$. The function $f$ extends to a continuous
function on $\overline\D$ such that $f(\cis{\bcdot})$ is log-Dini continuous if and only
if there exists a positive non-increasing function $\Phi : [0, r_0)\lrarw [0, +\infty]$, for
some $r_0\in (0, 1)$, such that:
\begin{itemize}
 \item[$a)$] $(\log(1/\bcdot))^n\Phi$ is of class $\leb{1}([0,r_0))$ for each $n\in \N$; and\vspace{1mm}
 \item[$b)$] $|f^\prime(r\cis{\tht})|\leq \Phi(1-r)$ for all $r\in (1-r_0, 1)$ and for every $\tht\in \R$.
\end{itemize}
Furthermore, whenever this happens, $\omo{\bv{f}}(\delta)$ is  dominated by $3\big(\int_{0}^\delta\Phi(x)dx\big)$
for $0\leq \delta< r_0$.
\end{theorem}
\begin{proof}
We first assume the existence of a $\Phi$ for which $(a)$ and $(b)$ hold true. From the paragraph following the
statement of Proposition~\ref{P:HL}, we have that $f$ extends to a continuous function of on $\bdy\D$. We
shall continue to denote this extension by $f$ and set $\bv{f} := \left.f\right|_{\bdy\D}$. We shall now deduce
some information about $\omo{\bv{f}}$.
\smallskip

To this end, fix
$\tht_1, \tht_2\in \R$ such that, for the moment, $0< \tht_2 - \tht_1< \pi$.
Let $\pth{r,\,\rho}: [0,1]\to \D$ be a path from
$r\cis{\tht_1}$ to $r\cis{\tht_2}$ whose image comprises the radial segments
$[r\cis{\tht_1}, \rho\cis{\tht_1}]$ and $[\rho\cis{\tht_2}, r\cis{\tht_2}]$, where $1-r_0< \rho< r< 1$,
and the (shorter) arc of the circle $\{z\in \C: |z| = \rho\}$ from
$\rho\cis{\tht_1}$ to $\rho\cis{\tht_2}$. Owing to the holomorphicity of $f^\prime$, and the
existence of the limits in \eqref{E:bLim}, we have
\[
 f(r\cis{\tht_2}) - f(r\cis{\tht_1})\,=\,\lim_{r\to 1^-}\int_{\pth{r,\,\rho}}\!f^{\prime}(z)dz
\]
for any $\rho$ as above. (A point about notation: each integral below of the
form $\int_{a}^{b}$, $a < b\in \R$, that involves a {\em non-negative function}
denotes the standard (hence {\em unoriented}) Lebesgue integral on the interval $[a, b]$.) Thus \vspace{-1mm}
\[
 |f(\cis{\tht_2}) - f(\cis{\tht_1})| \ \leq \ 2\int_{\rho}^{1}\!\Phi(1-s)ds + 
 \int_{\tht_1}^{\tht_2}\!\Phi(1-\rho)d\tht.
\]
Let us write $I(x) := \int_0^x\Phi(s)ds$. From the last inequality, we can see that our temporary restriction
on $(\tht_1, \tht_2)$ does not matter, and we have
\begin{align}
 |f(\cis{\tht_2}) - f(\cis{\tht_1})|\,&\leq\,2I(1-\rho) + |\tht_2 - \tht_1|\Phi(1-\rho) \notag \\
 &\leq\,\left(2 + \frac{|\tht_2 - \tht_1|}{1-\rho}\right)I(1-\rho) \ \ \ \forall \tht_1,\,\tht_2 : 
 |\tht_2 - \tht_1|\leq \pi. \label{E:inter1}
\end{align}
The second inequality follows from the fact that $\Phi$ is non-increasing. Let us define
\[
 \varpi(\delta)\,:=\,\begin{cases}
 					3I(\delta), &\text{if $0\leq \delta< r_0$}, \\
 					2\sup|\bv{f}|, &\text{if $r_0\leq \delta\leq \pi$}.
 					\end{cases}
\]
We shall make a choice for the free parameter $\rho$ in \eqref{E:inter1} based on 
$(\tht_1, \tht_2)$. Taking $\rho = 1-|\tht_2 - \tht_1|$ 
whenever $0< |\tht_2 - \tht_1|< r_0$, we see that
\begin{equation}\label{E:majorant}
 |f(\cis{\tht_2}) - f(\cis{\tht_1})|\,\leq\,\varpi(|\tht_2 - \tht_1|)
 \ \ \ \forall \tht_1,\,\tht_2 : |\tht_2 - \tht_1|\leq \pi.
\end{equation}
It is evident from \eqref{E:majorant} that $\varpi$ is a majorant of $\omo{\bv{f}}$.
\smallskip

From this last statement, it follows that to establish that $\bv{f}$ is log-Dini continuous,
it suffices to establish the inequality in \eqref{E:logDini} with
\begin{itemize}
 \item $\omo{g}$ replaced by $\varpi$, and
 \item the integral over the interval $[0, 1]$ replaced by an integral over the interval $[0, R]$ for
 some small $R > 0$ (since the integrand in \eqref{E:logDini} is unbounded only as $x\to 0^+$).
\end{itemize}   
Let us fix some $R$ such that $0< R< r_0$. By Tonelli's theorem, we have:
\begin{align*}
 \frac{1}{3}\int_{0}^R(\log(1/x))^n\,\frac{\varpi(x)}{x}\,dx\,&=\,\int_0^R\int_0^x
 		(\log(1/x))^n\,\frac{\Phi(s)}{x}\,ds\,dx \\
 		&=\,\int_0^R\left[\,\int_s^R\frac{(\log(1/x))^n}{x}\,dx\right]\Phi(s)\,ds \\
 		&=\,\frac{1}{n+1}\int_0^R\big[(\log(1/s))^{n+1} - (\log(1/R))^{n+1}\big]\,\Phi(s)\,ds\,<\,\infty
\end{align*}
for each $n\in \N$. The last assertion is just the condition $(a)$. From our remarks above, it follows
that $\bv{f}$ is log-Dini continuous.
\smallskip

Conversely, let us assume that $f$ extends to a continuous function on $\overline\D$ and that $\bv{f}$ is
log-Dini continuous. Then, for any $r\cis{\phi}\in \D$, taking $\phi\in (-\pi, \pi]$, the Cauchy integral theorem
gives us
\[
 f^\prime(r\cis{\phi})\,=\,\frac{1}{2\pi}\int_{-\pi}^\pi\!
 \frac{f(\cis{\tht})}{(\cis{\tht} - r\cis{\phi})^2}\cis{\tht}\,d\tht\,=\,\frac{1}{2\pi}\int_{-\pi}^\pi\!
 \frac{f(\cis{\tht}) - f(\cis{\phi})}{(\cis{\tht} - r\cis{\phi})^2}\cis{\tht}\,d\tht.
\]
Setting $\tau := \tht - \phi$, whence $|f(\cis{\tht}) - f(\cis{\phi})|$ transforms to
$|f(\cis{(\phi+\tau)}) - f(\cis{\phi})|$, we have
\[
 |f^\prime(r\cis{\phi})|\,\leq\,\frac{1}{\pi}\int_{0}^\pi
 \frac{\omo{\bv{f}}(\tau)}{r^2-2r\cos\tau+1}\,d\tau \ \equiv \ \Phi(1-r).
\]
We would be done if we could find some small $R> 0$ such that 
$\int_{0}^R(\log(1/x))^n\Phi(x)dx < \infty$ for each $n\in \N$.
\smallskip

Note that
\begin{align*}
 r^2-2r\cos\tau+1=(1-r)^2 + 4r\sin^2(\tau/2)&\geq (1-r)^2+4r(\tau/\pi)^2 \\
 &\geq (1-r)^2 + (\tau/\pi)^2 \ \ \forall r\in (1/4, 1),\ \forall \tau\in [0, \pi].
\end{align*}
Thus, if we fix $R\in (0, 3/4)$ and set $x := (1-r)$, then it suffices to establish the convergence of the
following integrals:
\[
 I_n\,:=\,\frac{1}{\pi}\int_{0}^R(\log(1/x))^n\left[\,\int_{0}^\pi\frac{\omo{\bv{f}}(\tau)}{x^2+(\tau/\pi)^2}\,d\tau\right]dx,
 \ \ n\in \N.
\]
Fix a $\delta\in (0, R)$ and let $I_n(\delta)$ denote the integral over $[\delta, R]$ of the integrand
on the right-hand side above. Making the change of variable $y := \tau/\pi{x}$ in the inner integral, we get
\begin{align}
 I_n(\delta)\,&=\,\int_{\delta}^R\frac{(\log(1/x))^n}{x}\left[\,\int_{0}^{1/x}
 			\frac{\omo{\bv{f}}(\pi xy)}{1+y^2}\,dy\right]dx \notag \\
 			&\leq\,\int_{\delta}^R(\log(1/x))^n\frac{\omo{\bv{f}}(x)}{x}\left[\,\int_{0}^{1/x}
 			\frac{1+\pi{y}}{1+y^2}\,dy\right]dx \notag \\
 			&\leq\,C\int_{\delta}^R(\log(1/x))^n\frac{\omo{\bv{f}}(x)}{x}\big[1 + \log(1/x)\big]\,dx,
 			\label{E:unifFin}
\end{align}
where $C> 0$ is a constant that does not depend on $x$, $n$ or $\delta$.
The first inequality above follows from the standard inequality
$\omo{\bv{f}}(\lambda x) \leq (\lambda +1)\omo{\bv{f}}(x)$ for all $\lambda\geq 0$ (and sufficiently
small that $\omo{\bv{f}}(\lambda x)$ makes sense); see \cite[Chapter 2, \S\,6]{devoreLorentz:ca93}.
Since $\bv{f}$ is log-Dini continuous, the integrands in \eqref{E:unifFin} are, in fact,
of class $\leb{1}([0, R])$ for each $n\in \N$. Thus, it follows from the above estimate that
$I_n< \infty$ for each $n\in \N$.
\smallskip

The final assertion of the theorem has already been established in the argument leading from the
inequality \eqref{E:inter1} to the inequality \eqref{E:majorant}.  
\end{proof}

One may ask whether there are any functions $f$ in the class $\hol(\D)\cap \smoo(\overline\D)$ beyond
those already described by the classical Hardy--Littlewood lemma for which $\omo{\bv{f}}$ is as in the above
theorem. We address this question by the following:

\begin{example}\label{Ex:less_reg}
There exist functions $f\in \hol(\D)\cap \smoo(\overline\D)$ such that $\bv{f}$ is
log-Dini continuous but belongs to no H{\"o}lder class.
\end{example}

\noindent{Pick a function $\psi\in \smoo(\bdy\D;\,\R)$ such that
\[
\omo{\psi}(\delta)\,=\,\exp\left(-C(\log(1/\delta))^{1-\eps}\right) \ \ \text{for $\delta$ close to $0$},
\]
where $C> 0$ and $0< \eps< 1$, and $\psi$ is not in any H{\"o}lder class. It is an elementary
exercise to show that
\begin{equation}\label{E:real-part}
 \int_{0}^1(\log(1/x))^n\,\frac{e^{\pw{-C(\log(1/x))^{1-\eps}}}}{x}\,dx\,<\,\infty
 \ \ \ \text{for} \ n = 0, 1, 2,\dots 
\end{equation}
Let $\cnj{\psi}$ denote the conjugate function of $\psi$. Then, by the Privalov--Zygmund estimate,
there exists a constant $K> 0$ such that
\begin{equation}\label{E:P-Z}
 \omo{\cnj{\psi}}(\delta)\,\leq\,K\left[\,\int_{0}^\delta\frac{\omo{\psi}(x)}{x}\,dx +
 \delta\int_{\delta}^\pi\frac{\omo{\psi}(x)}{x^2}\,dx\right].
\end{equation}
This calls for a somewhat careful estimation of the two integrals above. First, by
decomposing $(0,\delta]$ as $(0,\delta] = \cup_{j=0}^\infty[2^{-(j+1)}\delta, 2^{-j}\delta]$ (by
\eqref{E:real-part}, the integral of $\omo{\psi}(x)/x$ over 
 $(0, \delta]$ is the same as the integral over $[0, \delta]$), we have
\begin{align}
 \int_{0}^\delta\frac{e^{\pw{-C(\log(1/x))^{1-\eps}}}}{x}\,dx\,&\leq\,\sum_{j=0}^\infty
 \exp\left\{\!-C\left(\log\left(\frac{2^j}{\delta}\right)\right)^{1-\eps}\right\}. \notag \\
 &\leq\,\sum_{j=0}^\infty
 \exp\big[\!-\!2^{-\eps}C\left((j\log2)^{1-\eps}+(\log(1/\delta))^{1-\eps}\right)\big] \notag \\
 &=\,K^\prime\exp\left(-2^{-\eps}C(\log(1/\delta)^{1-\eps}\right) \ \ \ \text{for $\delta$ close to $0$}. \label{E:conj-1}
\end{align}
The second inequality above arises from the fact that, close to $0$, $\omo{\psi}$ is a concave function.
Let $R> 0$ be such that $\omo{\psi}$ is concave on $[0, R]$. Then, since $\omo{\psi}(0) = 0$,
$x\!\longmapsto\!\omo{\psi}(x)/x$ is a decreasing function on $[0, R]$. Using this fact, we
get the crude, but adequate, estimate
\begin{equation}\label{E:conj-2}
 \delta\int_{\delta}^\pi\frac{e^{\pw{-C(\log(1/x))^{1-\eps}}}}{x^2}\,dx\,\leq\,\log\left(\frac{R}{\delta}\right)
 e^{\pw{-C(\log(1/\delta))^{1-\eps}}} + O(\delta).
\end{equation}

It is easy to see directly from the estimate \eqref{E:P-Z} that $\lim_{\delta\to 0^+}\omo{\cnj{\psi}}(\delta) = 0$.
Thus, we conclude from \eqref{E:limit0} that $\cnj{\psi}\in \smoo(\bdy\D;\,\R)$. Then, by definition,
$(\psi+i\cnj{\psi})$ is the boundary value of a function $\Psi\in \hol(\D)\cap\smoo(\overline\D)$. Now
\eqref{E:real-part}, taken together with the estimates \eqref{E:conj-1} and \eqref{E:conj-2}, implies
that $\bv{\Psi}$ is log-Dini continuous. However, since $\psi = \bv{(\re\Psi)}$ was chosen so that it
does {\em not} belong to any H{\"o}lder class, $\bv{\Psi}$ is not in any H{\"o}lder class either. \hfill
$\blacktriangleleft$

\begin{remark}
With more intensive analysis, one can show that the holomorphic functions constructed in
Example~\ref{Ex:less_reg} are such that the modulus of continuity of $\bv{f}$ is
$O(\omo{\psi})$. That \eqref{E:conj-1} can be improved is not hard to see: one uses a better
lower bound for $(\log(2^j) + \log(1/\delta))^{1-\eps}$ in the step preceding \eqref{E:conj-1}. The
estimate \eqref{E:conj-2} is rather crude. It can be improved as desired, but this requires some
effort. Since this is not the main thrust of the present section, we shall not elaborate any further
on the last point.
\end{remark}
 
\section{The proof of Theorem~\ref{T:contEXT}}\label{S:contEXT}

We now return to several complex variables.
We begin by presenting some notation. If $\OM$ is a bounded domain in $\C^n$, $z\in \OM$ and
$v\in \C^n\setminus\{0\}$, then
\begin{align*}
 d_{\OM}(z)\,:=\,&\text{the Euclidean distance of $z$ from $\bdy\OM$}, \\
 r_{\OM}(z; v)\,:=\,&\text{the radius of the largest complex-affine closed disc, centered at $z$} \\
 				&\text{and tangent to $v$, that is contained in $\overline{\OM}$}.
\end{align*}
The crux of the proof of Theorem~\ref{T:contEXT} is to find an estimate for $f^\prime$,
where $f: \D\to \OM$ is a complex geodesic. One could then try to apply Proposition~\ref{P:HL}
to deduce continuous extension to $\bdy\D$. If we have a reasonably good estimate for the Kobayashi
pseudometric on $\OM$, where $\OM$ is as in Theorem~\ref{T:contEXT}, then we can use it
to estimate $f^\prime$. This explains the need for the following result of Graham:

\begin{result}[Graham, \cite{graham:dthmcd90}]\label{R:kob}
Let $\OM$ be a bounded convex domain in $\C^n$ and let $\kappa_{\OM}(z;\,\bcdot)$ denote the Kobayashi
metric on $\OM$ at the point $z\in \OM$. Then:
\[
 \frac{\|v\|}{2r_{\OM}(z; v)}\,\leq\,\kappa_{\OM}(z; v)\,\leq\,\frac{\|v\|}{r_{\OM}(z; v)} \quad
 \forall z\in \OM \ \text{and} \ \forall v\in \C^n.
\]
\end{result}
\noindent{We must point out that in this section all distances and norms on $\C^n$ will be the Euclidean
distance and the Euclidean norm, both denoted by $\|\bcdot\|$.}
\smallskip

To carry out the programme sketched above, we will require explicit estimates on $r_{\OM}$. This
is the role of the following lemma:  
 
\begin{lemma}\label{L:rEST}
Let $\OM$ be as in Theorem~\ref{T:contEXT}. There exists a compact subset $K$ of $\OM$ such that
for each $z\in \OM\setminus K$,
\[
 r_{\OM}(z; v)\,\leq\,2\crux{d_{\OM}(z)} \quad
 \forall v\in \C^n\setminus\{0\},
\]
where $C$ and $\alpha$ are the constants appearing in Theorem~\ref{T:contEXT}.
\end{lemma}
\begin{proof}
Define $x_0$ as follows (it is not hard to argue that the set on the right is finite):
\begin{equation}\label{E:cap}
 x_0\,:=\,\min\left[\{C/2\}\cup \left\{x\in (0,C) : x = [\,\log(C/x)\,]^{-1/\alpha}\right\}\right].
\end{equation}
Let $s_0$ and $R_0$ be as given by Definition~\ref{D:supp} with $F = C\flt{\alpha}(\,\|\bcdot\|\,)$, and write
\[
 \suppmf\,:=\,\{z=(\z, z_n)\in \C^n : \im(z_n) = C\flt{\alpha}(\|\z\|)\}.
\]
We can find a compact subset $K$ of $\OM$ such that whenever $z\in \OM\setminus K$,
\begin{itemize}
 \item $d_{\OM}(z) < \min(s_0, x_0)$; and
 \item For any point $w(z)\in \bdy\OM$ that satisfies $d_{\OM}(z) = \|z - w(z)\|$, every complex
 line of the form
 \[
  \uni^{w(z)}(z + \C{v}) \; \;(\,=\,(0,\dots,0,i\bcdot d_{\OM}(z)) + \C{\uni_{w(z)}(v)}\;), \; \;
  v\in H_{w(z)}(\bdy\OM)\setminus\{0\},
 \]
 intersects $\suppmf$ in a circle of radius $[\,\log(C/d_{\OM}(z))\,]^{-1/\alpha} = 
 (C\flt{\alpha})^{-1}(d_{\OM}(z))$.
\end{itemize}
Here $\uni^{w(z)}$ is as described in Definition~\ref{D:supp}.
For each $z\in \OM\setminus K$, let us {\em fix} a $w(z)$ for the remainder of this proof.
The above implies that if $z\in \OM\setminus K$, then
\begin{equation}\label{E:ctang}
  r_{\OM}(z; v)\,\leq\,\crux{d_{\OM}(z)} \quad 
  \forall v\in H_{w(z)}(\bdy\OM)\setminus\{0\}.
\end{equation}
Of course, for all such $z$, we also have
\begin{equation}\label{E:perp}
 r_{\OM}(z; v)\,=\,d_{\OM}(z) \quad 
  \forall v\in H_{w(z)}(\bdy\OM)^\perp\setminus\{0\}.
\end{equation}
(The orthogonal complement here is with respect to the standard
Hermitian inner product on $\C^n$.)
\smallskip

Fix a $z\in \OM\setminus K$, consider a general {\em unit} vector $v\in \C^n\setminus\{0\}$, and let $\tht_v\in \R$ be
such that
\[
 \cis{\tht_v}\uni_{w(z)}v\,=\,(\cis{\tht_v}(\uni_{w(z)}v)^\prime, -i|(\uni_{w(z)}v)_n|)\,=:\,V.
\]
Observe that $r_{\OM}(z; v) = r_{\OM}(z; \cis{\tht_v}v)$. In view of \eqref{E:ctang} and \eqref{E:perp}, we may
focus on those $V$\,---\,writing $V = (\V, -iV_n)$\,---\,such that $\V \neq 0$ and $V_n\ > 0$.
We view a portion of $\bdy\OM$ around $w(z)$ after the application of the $\C$-affine transformation
$\uni^{w(z)}$. It then follows from elementary coordinate geometry that if $\rho$ is a positive number that
satisfies (recall that $z$ is mapped to $(0,\dots,0,i\bcdot d_{\OM}(z))$ under $\uni^{w(z)}$) 
\begin{equation}\label{E:triangle}
 C\flt{\alpha}(\rho) + \rho\frac{V_n}{\|\V\|}\,=\,d_{\OM}(z),
\end{equation}
then the set $(0,\dots,0,i\bcdot d_{\OM}(z)) + \overline{D(0;\rho_*)}\,V$ intersects $\suppmf$, where
\[
 \rho_*\,:=\,\rho/\|\V\|.
\]
Since
\[
 \suppmf\,\subset\,B^{n-1}(0;R_0)\times((-s_0,s_0)+i(-s_0,s_0))\setminus \uni^{w(z)}(\OM),
\]
it easily follows that $(0,\dots,0,i\bcdot d_{\OM}(z)) + \overline{D(0;\rho_*)}\,V$ intersects
$\uni^{w(z)}(\bdy\OM)$. Hence
\begin{equation}\label{E:interim}
 r_{\OM}(z; v)\,=\,r_{\uni^{w(z)}(\OM)}((0,\dots,0,i\bcdot d_{\OM}(z)); V)\,\leq\,\rho_*.
\end{equation}

By \eqref{E:triangle} and the fact that $\flt{\alpha}$ is increasing on $(0,\infty)$ we get
that $\rho \leq [\,\log(C/d_{\OM}(z))\,]^{-1/\alpha}$. Therefore,
\[
 \rho_* \ \leq\,\crux{d_{\OM}(z)} +\,d_{\OM}(z) \ \leq \ 2\crux{d_{\OM}(z)}.
\]
The second inequality follows from the fact that $d_{\OM}(z) < \min(s_0, x_0)$,
where $x_0$ is defined by \eqref{E:cap}.
The above inequality, together with \eqref{E:interim} and the estimates in the first paragraph
of this proof, gives the desired conclusion.
\end{proof}

A key requirement of our proof is to transcribe an estimate for $\|f^\prime(\zt)\|$ (here $f: \D\to \OM$ is
a complex geodesic) given in terms of
$\kappa_{\OM}(f(\zt); f^\prime(\zt))$\,---\,which is provided by Graham's result\,---\,into an estimate given
in terms of $\zt\in \D$. One such tool is an estimate by Lempert \cite[Proposition~12]{lempert:mKrdb81}.
However, since our domains of interest are not strongly convex, we will need an extension of this
estimate. This has been provided by Mercer, and is as follows:

\begin{result}[Mercer, \cite{mercer:cgihmcd}]\label{R:Mer}
Let $\OM$ be a bounded convex domain in $\C^n$ and let $f: \D\to \OM$ be a complex geodesic.
There exists a constant $\beta > 1$ and constants $C_1, C_2 > 0$ such that
\[
 C_1(1 - |\zt|)\,\leq\,d_{\OM}(f(\zt))\,\leq\,C_2(1 - |\zt|)^{1/\beta} \quad
 \forall \zt\in \D.
\]
\end{result}

We are now in a position to give
\smallskip

\begin{proof}[The proof of Theorem~\ref{T:contEXT}]
Let $f: \D\to \OM$ be a complex geodesic. It is easy to argue that $f$ is a proper map: see
\cite[Proposition~4.6.3]{kobayashi:hcs98}, for instance. Let $K$ be the compact set given by
Lemma~\ref{L:rEST}. By the properness of $f$, there exists a constant $r_0 > 0$ such that
$f(\zt)\in \OM\setminus K$ whenever $1-r_0 < |\zt| < 1$. By Result~\ref{R:kob}
\[
 \|f^\prime(\zt)\|\,\leq\,2r_{\OM}(f(\zt); f^\prime(\zt))\,\kappa_{\OM}
 (f(\zt); f^\prime(\zt))\,=\,\frac{2r_{\OM}(f(\zt); f^\prime(\zt))}{1-|\zt|^2} \quad
 \forall \zt\in \D.
\]
The equality between the second and the third expression is due to the fact that $f$ is a complex geodesic.
For $\zt\in \D$ such that $1-r_0 < |\zt| < 1$, it follows from Lemma~\ref{L:rEST} and the
above inequality that
\[
 \|f^\prime(\zt)\| \ \leq \ 4\crux{d_{\OM}(f(\zt))}\!\!\frac{1}{1-|\zt|^2} \
 \leq \ 4\crux{d_{\OM}(f(\zt))}\!\!\frac{1}{1-|\zt|}.
\]
Finally, we invoke Mercer's estimate, Result~\ref{R:Mer}, to get
\begin{align}
 \|f^\prime(\zt)\| \ &\leq \ 4\CruxII{(1-|\zt|)^{1/\beta}}{C/C_2}\!\!\frac{1}{1-|\zt|} \notag \\
 &\equiv K_1\CruxII{1-|\zt|}{K_2}\!\!\frac{1}{1-|\zt|} \quad \forall \zt: 1-r_0 < |\zt| < 1, \label{E:derivEST}
\end{align}
where $K_1$ and $K_2$ are appropriate positive constants. Define a function $\Phi: [0, r_0)\to [0, +\infty]$ as follows:
\[
 \Phi(x)\,:=\,\begin{cases}
 			\frac{K_1}{x}\CruxII{x}{K_2}, &\text{if $0 < x < r_0$}, \\
 			+\infty, &\text{if $x=0$}.
			\end{cases}
\]
Write $f = (f_1,\dots,f_n)$. By \eqref{E:derivEST}, each component satisfies
\[
 |f_j^\prime(\zt)|\,\leq\,\Phi(1-|\zt|) \quad\forall\zt : 1-r_0 < |\zt| < 1, \;  j=1,\dots,n.
\]
Given Proposition~\ref{P:HL}, the theorem will follow if we can show that
the above $\Phi$ is of class $\leb{1}([0,r_0))$. The convergence of the
integral of $\Phi$ is a standard example; it converges precisely when $0< \alpha< 1$. Thus,
by our assumption on $\alpha$, $f$ extends continuously up to $\bdy\D$.
\end{proof}

\end{document}